\documentclass[10pt,reqno]{amsart}
\usepackage[hypertex]{hyperref}

\usepackage{amsmath,amsthm}
\usepackage{amssymb}
\usepackage{euscript}
\usepackage[frame,ps,matrix,arrow,curve,rotate,all,2cell,tips]{xy}
\usepackage{epic,eepic}

\numberwithin{equation}{subsection}

\newcommand{\sqsp}{\renewcommand{\baselinestretch}{1.2}\tiny\normalsize}

\newtheorem{theorem}[subsection]{Theorem}
\newtheorem{lemma}[subsection]{Lemma}
\newtheorem{proposition}[subsection]{Proposition}
\newtheorem{corollary}[subsection]{Corollary}

\theoremstyle{definition}
\newtheorem{definition}[subsection]{Definition}
\newtheorem{example}[subsection]{Example}

\newcommand{\bk}{\mathbf{k}}

\newcommand{\bZ}{\mathbf{Z}}
\newcommand{\unit}{\mathbf{1}}

\newcommand{\cS}{\mathcal{S}}

\newcommand{\fC}{\mathfrak{C}}
\newcommand{\sltwo}{\mathfrak{sl}_2}

\newcommand{\ainfty}{\mathsf{A}_{\infty}}
\newcommand{\As}{\mathsf{As}}
\newcommand{\Bi}{\mathsf{Bi}}
\newcommand{\sE}{\mathsf{E}}
\newcommand{\linfty}{\mathsf{L}_{\infty}}
\newcommand{\Nam}{\mathsf{Nam}}
\newcommand{\sO}{\mathsf{O}}
\newcommand{\sP}{\mathsf{P}}
\newcommand{\sQ}{\mathsf{Q}}
\newcommand{\ugr}{\mathsf{UGr}}
\newcommand{\YBE}{\mathsf{YBE}}

\newcommand{\ybar}{\overline{y}}

\newcommand{\andspace}{\quad\text{and}\quad}
\newcommand{\bigo}{\bigodot}
\newcommand{\bracket}{[...]}

\newcommand{\profilev}{(|out(v)|,|in(v)|)}
\newcommand{\nicearrow}{\SelectTips{cm}{10}}
\newcommand{\signsigma}{(-1)^{\epsilon(\sigma)}}

\DeclareMathOperator{\Hom}{Hom}
\DeclareMathOperator{\Aut}{Aut}

\newcommand{\corolla}{{\unitlength=1mm
\begin{picture}(13,17)(-2,-2)
\put(5,5){\makebox(0,0){$\bullet$}}
\put(0,0){\vector(1,1){4.5}}
\put(10,0){\vector(-1,1){4.5}}
\put(5,5){\vector(1,1){5}}
\put(5,5){\vector(-1,1){5}}
\put(5,1){\makebox(0,0){...}}
\put(5,9){\makebox(0,0){...}}
\end{picture}}}

\newcommand{\gcorolla}{{\unitlength=.6mm
\begin{picture}(40,55)(-20,-15)
\put(0,0){\makebox(0,0){$\bullet$}}
\put(0,20){\makebox(0,0){$\bullet$}}
\put(0,20){\vector(1,1){10}}
\put(0,20){\vector(-1,1){10}}
\put(-10,-10){\vector(1,1){10}}
\put(10,-10){\vector(-1,1){10}}
\put(0,27){\makebox(0,0){...}}
\put(0,-8){\makebox(0,0){...}}
\put(0,10){\makebox(0,0){...}}
\put(-10,20){\makebox(0,0){$w$}}
\put(-10,0){\makebox(0,0){$v$}}
\qbezier(0,0)(-12,8)(-2,18)
\put(-2,18){\vector(1,1){2}}
\qbezier(0,0)(12,8)(2,18)
\put(2,18){\vector(-1,1){2}}
\end{picture}}}

\begin{document}

\title{Hom-algebras via PROPs}
\author{Donald Yau}

\begin{abstract}
Hom-algebras over a PROP are defined and studied.  Several twisting constructions for Hom-algebras over a large class of PROPs are proved, generalizing many such results in the literature.  Partial classification of Hom-algebras over a PROP is obtained.
\end{abstract}



\address{Department of Mathematics\\
    The Ohio State University at Newark\\
    1179 University Drive\\
    Newark, OH 43055, USA}
\email{dyau@math.ohio-state.edu}

\date{\today}
\maketitle

%\tableofcontents

\sqsp

%%%%%%%%%%%%%%%%%%%%%%
\section{Introduction}
%%%%%%%%%%%%%%%%%%%%%%

A \emph{Hom-Lie algebra} is a generalization of a Lie algebra in which the anti-symmetric bracket satisfies the Hom-Jacobi identity,
\[
[[x,y],\alpha(z)] + [[z,x],\alpha(y)] + [[y,z],\alpha(x)] = 0,
\]
in which $\alpha$ is a linear self-map, called the twisting map.  Lie algebras are Hom-Lie algebras with identity twisting map.  Hom-Lie algebras were introduced in \cite{hls} to describe certain deformations of the Witt and the Virasoro algebras.  Hom-type generalizations of many classical algebraic structures have been studied, including Hom-associative algebras \cite{ms1,ms3,ms4,yau1,yau2}, $n$-ary Hom-Nambu algebras \cite{ams1,ams2,ams3,yau7,yau8,yau9}, Hom-Novikov algebras \cite{yau5}, Hom-bialgebras \cite{ms2,yau4}, and the Hom-Yang-Baxter equation \cite{yau3,yau6}, among many others.

There are two common themes among all these different types of Hom-algebras, namely, their definitions and twisting construction results.  Let us discuss each item.

A Hom-type generalization of a given kind of algebras is defined by strategically replacing the identity map in the defining axioms by general twisting maps. For example, the Jacobi identity involves iterated composition of the form $[[x,y],z]$.  In the Hom-Jacobi identity, this iterated composition is replaced by $[[x,y],\alpha(z)]$.  Conceptually one can think of the twisting map $\alpha$ as balancing the number of operations applied to the elements involved.  Indeed, in $[[x,y],z]$ both $x$ and $y$ are act upon twice by the bracket, whereas $z$ is only act upon once by the bracket.  In the Hom-type analogue $[[x,y],\alpha(z)]$, the element $z$ is now act upon twice, once by each of the twisting map and the bracket.

Another common theme among different types of Hom-algebras is that they can be constructed from classical algebras by twisting along an endomorphism.  The author first observed in \cite{yau2} that if $A$ is a $G$-associative algebra \cite{gr} (such as an associative or a Lie algebra) with multiplication $\mu$ and if $\alpha \colon A \to A$ is an algebra morphism, then $A$ becomes a $G$-Hom-associative algebra with multiplication $\alpha \circ \mu$ and twisting map $\alpha$.  Many examples of Hom-algebras can be obtained using this twisting construction.  See, for example, \cite{yau2,yau5,yau6} for Hom-associative, Hom-Lie, and Hom-Novikov algebras constructed using this method.  This twisting construction has been adapted to most other kinds of Hom-algebras and is the main tool in constructing examples of Hom-algebras.

From the above discussion, it is natural to ask if there is a general framework that underlies both the definitions and the twisting construction results for different kinds of Hom-algebras.  The purpose of this paper is to describe such a general framework for defining and studying Hom-algebras.  Before we describe the content of this paper, let us briefly discuss what this general framework is.

To describe different types of algebras and their Hom-type analogues, we use Mac Lane's general machinery of PROPs \cite{maclane}.  The name PROP stands for \emph{product and permutation category}.  Briefly, a PROP is a bigraded module with two operations that satisfy some axioms.  There are algebras over any given PROP $\sP$ with various operations coming from $\sP$.  To obtain a particular kind of algebras, one specifies an appropriate PROP.   For example, if $\sP$ is the bialgebra PROP, then $\sP$-algebras are precisely bialgebras.  There are other PROPs for $n$-ary Nambu algebras, solutions of the Yang-Baxter equation, and so forth.

There are several advantages of the general PROP setting.  First, every result about PROPs applies to many different kinds of algebras by specifying to particular PROPs.  Second, PROPs can describe algebraic structures with multiple inputs and multiple outputs simultaneously.  This is necessary, for example, in the case of (Hom-)bialgebras and the (Hom-)Yang-Baxter equation.  Third, as general as a PROP is, the definition of a PROP is quite simple.  Once the notations are established, it is not much more difficult to work with PROPs than with, say, bialgebras.

The rest of this paper is organized as follows.  In the next section, the definitions of a PROP, an algebra over a PROP, and free PROPs are recalled.  In section \ref{sec:homalgebra} Hom-algebras over an arbitrary PROP are defined.  There are actually many different kinds of Hom-algebras over a given PROP.  This variety of Hom-algebras arises from the fact that there are choices as to which identity maps are replaced by twisting maps and how many twisting maps are used. In section \ref{sec:examples} some specific types of Hom-algebras over PROPs are discussed.

In section \ref{sec:twit} several construction results and partial classification for Hom-algebras are proved.  In Theorem \ref{thm:twist} it is shown that for a large class of PROPs, called normal PROPs, every Hom-algebra gives rise to another one whose structure maps are twisted by an endomorphism.  Most of the twisting constructions for Hom-algebras in the literature are special instances of Theorem \ref{thm:twist}.  There are several useful special cases of this general twisting result.  First, over a normal PROP $\sP$, every multiplicative Hom-$\sP$-algebra yields a sequence of multiplicative Hom-$\sP$-algebras whose structure maps are twisted repeatedly by its own twisting map (Corollary \ref{cor1:twist}).  On the other hand, starting with a $\sP$-algebra and an endomorphism $\beta$, the twisting result yields a Hom-$\sP$-algebra whose structure maps are twisted by $\beta$ and whose twisting maps are all given by $\beta$ (Corollaries \ref{cor2:twist} and \ref{cor3:twist}).  Hom-$\sP$-algebras arising from $\sP$-algebras via this twisting construction are partially classified in Corollaries \ref{cor:classification} and \ref{cor2:classification}.

%%%%%%%%%%%%%%%%%%%%%%%%%%%%
\section{PROPs and algebras}
\label{sec:propalgebra}
%%%%%%%%%%%%%%%%%%%%%%%%%%%%

In this preliminary section, we recall some basic definitions and facts regarding PROPs and their algebras.  The references for most of this section are \cite{maclane,markl1}.  There are several equivalent characterizations of PROPs, but in this paper we only use one of them.  There are more conceptual descriptions of PROPs as algebras over certain monads \cite{fmy,markl1} or as $2$-monoids \cite{jy}.

%%%%%%%%%%%%%%%%%%%%%%%%
\subsection{Conventions}

The cardinality of a set $S$ is denoted by $|S|$.  Throughout this paper we work over a ground field $\bk$ of characteristic $0$.  Our definitions and results are stated for $\bk$-modules.  Their obvious analogues for $\bZ$-graded and differential graded $\bk$-modules are also valid.  Moreover, all the results remain true if we work over the semi-simple algebra $\bigoplus_{\fC} \bk$ for some non-empty set $\fC$.

% sigma bimodule

Let us first recall the underlying object of a PROP.  Let $\Sigma_n$ denote the symmetric group on $n$ letters.

\begin{definition}
A \emph{$\Sigma$-bimodule} is a bigraded $\bk$-module $P=\{P(n,m)\}_{n,m \geq 0}$ such that $P(n,m)$ is equipped with a left $\Sigma_n$-action and a right $\Sigma_m$-action that commute with each other.  An element in $P(n,m)$ is said to have \emph{biarity $(n,m)$}.  A \emph{morphism} $f \colon P \to Q$ of $\Sigma$-bimodules is a bigraded collection $f = \{f(n,m) \colon P(n,m) \to Q(n,m)\}$ of bi-equivariant maps.
\end{definition}

%There are several equivalent definitions of a PROP, and we will use two of them.  First let us recall the biased definition of a PROP.

% (colored) PROPs: biased definition
\begin{definition}
\label{def:prop}
A \emph{PROP} is a $\Sigma$-bimodule $\sP = \{\sP(n,m)\}$ that is equipped with
\begin{enumerate}
\item
a \emph{horizontal composition}
\begin{equation}
\label{hcomp}
\otimes \colon \sP(n,m) \otimes \sP(q,p) \to \sP(n+q,m+p),
\end{equation}
\item
a \emph{vertical composition}
\begin{equation}
\label{vcomp}
\circ \colon \sP(n,m) \otimes \sP(m,l) \to \sP(n,l),
\end{equation}
and
\item
a \emph{unit element} $\unit \in \sP(1,1)$,
\end{enumerate}
satisfying the following axioms.
\begin{enumerate}
\item
The horizontal and the vertical compositions are associative and (bi-)equivariant in the obvious sense.
\item
The \emph{interchange rule}
\begin{equation}
\label{interchange}
(a \otimes b) \circ (c \otimes d) = (a \circ c) \otimes (b \circ d)
\end{equation}
holds for $a,b,c,d \in \sP$ whenever the compositions all make sense.
\item
The unit element is a two-sided identity for the vertical composition, in the sense that
\[
a = a \circ \unit^{\otimes m} = \unit^{\otimes n} \circ a
\]
for all $a \in \sP(n,m)$.
\end{enumerate}
A \emph{morphism} of PROPs is a morphism of the underlying $\Sigma$-bimodules that preserves the horizontal compositions, the vertical compositions, and the unit elements.  An \emph{ideal} of a PROP is defined as a sub-$\Sigma$-bimodule with the usual closure properties with respect to both the horizontal and the vertical compositions.
\end{definition}

\begin{example}
Let $\fC$ be a non-empty set.  If we use modules over the semi-simple algebra $\bigoplus_{\fC} \bk$ instead of $\bk$, then the above definitions give a $\fC$-colored $\Sigma$-bimodule and a $\fC$-colored PROP \cite{fmy,jy}, respectively.
\end{example}

\begin{example}
\label{ex:operad}
If $\sP$ is a PROP, then
\[
U\sP = \{\sP(1,m)\}_{m \geq 0}
\]
is naturally an operad \cite{may}, in which the operad structure map is an iterated horizontal composition followed by a vertical composition in $\sP$.  The forgetful functor $U$ admits a left adjoint, which sends an operad $\sO$ to the free PROP $\sO_{prop}$ generated by it.  Moreover, the categories of algebras of $\sO$ and $\sO_{prop}$ (Definition \ref{def:propalgebra} below) are naturally equivalent \cite{jy} (section 9).  In particular, every result below about PROPs applies to operads as well.
\end{example}

\begin{example}
\label{ex:endprop}
The \emph{endomorphism PROP} of a $\bk$-module $V$ is the PROP $\sE_V$ specified as follows:
\begin{enumerate}
\item
The components of $\sE_V$ are
\[
\sE_V(n,m) = \Hom(V^{\otimes m}, V^{\otimes n}),
\]
with the $\Sigma_n$-$\Sigma_m$ action given by permutations of tensor factors.
\item
The horizontal composition \eqref{hcomp} on $\sE_V$ is given by tensor product of morphisms.
\item
The vertical composition \eqref{vcomp} on $\sE_V$ is given by the usual composition of morphisms.
\item
The unit element in $\sE_V$ is the identity map on $V$.
\end{enumerate}
\end{example}

% algebras over a PROP
\begin{definition}
\label{def:propalgebra}
Let $\sP$ be a PROP.  A \emph{$\sP$-algebra} is a morphism $\lambda \colon \sP \to \sE_V$ of PROPs, where $\sE_V$ is the endomorphism PROP of a $\bk$-module $V$ (Example \ref{ex:endprop}).  In this case, we also call $V$ a $\sP$-algebra, or an algebra over $\sP$, with structure map $\lambda$.
\end{definition}

By adjunction the structure map $\lambda$ of a $\sP$-algebra $V$ is equivalent to a collection of maps
\[
\left\{\lambda \colon \sP(n,m) \otimes V^{\otimes m} \to V^{\otimes n}\right\}_{m,n \geq 0}
\]
that are compatible with the horizontal and the vertical compositions, the unit elements, and the $\Sigma_n$-$\Sigma_m$ actions.  For $x \in \sP(n,m)$ we write
\[
\lambda(x) \colon V^{\otimes m} \to V^{\otimes n}
\]
for the map $\lambda(x,-)$.

\begin{definition}
\label{def:propalgmorphism}
Let $\sP$ be a PROP and $\lambda \colon \sP \to \sE_V$ and $\rho \colon \sP \to \sE_W$ be $\sP$-algebras.  Then a \emph{morphism of $\sP$-algebras} $f \colon V \to W$ is a linear map $f \colon V \to W$ of $\bk$-modules such that the square
\[
\nicearrow
\xymatrix{
V^{\otimes m} \ar[r]^{\lambda(x)} \ar[d]_{f^{\otimes m}} & V^{\otimes n} \ar[d]^{f^{\otimes n}}\\
W^{\otimes m} \ar[r]^{\rho(x)} & W^{\otimes n}
}
\]
commutes for all $x \in \sP(n,m)$.
\end{definition}

Next we discuss free PROPs. %and an unbiased definition of a PROP.
Some preliminary discussion of graphs is needed.

%%%%%%%%%%%%%%%%%%%%%%%%%%%%
\subsection{Directed graphs}

A \emph{directed $(n,m)$-graph} \cite{markl1} is a finite, not-necessarily connected, directed graph $G$ such that:
\begin{enumerate}
\item
$G$ does not have directed cycles.
\item
At each vertex $v$ in $G$, the set $in(v)$ of incoming edges and the set $out(v)$ of outgoing edges are separately labeled.
\item
$G$ has $m$ inputs, which are edges without an initial vertex, and $n$ outputs, which are edges without a terminal vertex.  The sets $in(G)$ and $out(G)$ of in/outputs of $G$ are separately labeled.
\end{enumerate}
The exceptional graph
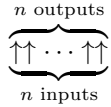
\[
\overbrace{\underbrace{\uparrow \uparrow \cdots \uparrow \uparrow}_{\text{$n$ inputs}}}^{\text{$n$ outputs}}
\]
with $n$ inputs, $n$ outputs, and no vertex is also allowed.  An \emph{isomorphism} of directed $(n,m)$-graphs is a bijection between the sets of edges and vertices that preserves all the labels and incidence relations at each vertex.  The category of directed $(n,m)$-graphs and their isomorphisms is denoted by $\ugr(n,m)$.

%%%%%%%%%%%%%%%%%%%%%%%
\subsection{Corollas}
%There are units for graph substitutions.
Define the \emph{$(n,m)$-corolla} as the directed $(n,m)$-graph $C_{n,m}$  with one vertex and no edges other than the $m$ inputs and the $n$ outputs.  The labels at the unique vertex coincide with the labels of the in/outputs.  The corolla $C_{n,m}$ can be graphically represented as:
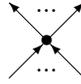
\[
\corolla
\]
The bottom $m$ arrows are the inputs, and the top $n$ arrows are the outputs.

%If $v$ is a vertex in a directed graph, define the corolla associated to $v$ as
%\[ C_v = C_{|out(v)|,|in(v)|}. \]
%Now if $G$ is a directed $(n,m)$-graph, then we have
%\[ G = C_{n,m}(G) = G(C_v). \]
%In other words, the set of corollas forms a two-sided identity for graph substitution.

%%%%%%%%%%%%%%%%%%%%%%%%%%%%%
\subsection{Decorated graphs}

Let $X$ be a $\Sigma$-bimodule and $G$ be a directed $(n,m)$-graph.  Define the object
\begin{equation}
\label{decoratedgraph}
X(G) = \bigotimes_{v \in G} X(|out(v)|,|in(v)|),
\end{equation}
in which the tensor product is taken over the set of vertices in $G$.  A typical element in $X(G)$ is denoted by
\begin{equation}
\label{xv}
\otimes_{v \in G} x_v
\end{equation}
with $x_v \in X\profilev$ and is called a \emph{decorated $(n,m)$-graph}.

%%%%%%%%%%%%%%%%%%%%%%%
\subsection{Free PROPs}

Let $X$ be a $\Sigma$-bimodule.  Define the $\Sigma$-bimodule $FX$ with
\begin{equation}
\label{freeprop}
FX(n,m) = \bigoplus_{[G] \in \ugr(n,m)} X(G)_{Aut(G)}.
\end{equation}
Here the direct sum is taken over the isomorphism classes of directed $(n,m)$-graphs, and $X(G)_{Aut(G)}$ is the module of coinvariants of $X(G)$ under the natural $Aut(G)$-action.  The $\Sigma_n$-$\Sigma_m$ action on $FX(n,m)$ comes from permutations of the output/input labels of $G \in \ugr(n,m)$.  Therefore, elements in $FX$ are (represented by) decorated graphs \eqref{xv}.

There is a natural PROP structure on $FX$:
\begin{enumerate}
\item
The horizontal composition \eqref{hcomp} on $FX$ is induced by disjoint unions of directed graphs.
\item
The vertical composition \eqref{vcomp} on $FX$ is induced by grafting of directed graphs, where the outputs of one are connected to the inputs of the other.
\item
The unit element $\unit \in FX$ is the element $1 \in \bk = X(\uparrow) \subseteq FX(1,1)$.
\end{enumerate}
This PROP $FX$ is called the \emph{free PROP generated by $X$}.  The functor $F$ is the left adjoint of the forgetful functor from PROPs to $\Sigma$-bimodules.

There is a natural map $X \to FX$ of $\Sigma$-bimodules given by
\begin{equation}
\label{corinclusion}
X(n,m) = X(C_{n,m}) \xrightarrow{\text{natural}} FX(n,m).
\end{equation}
We call this map the \emph{corolla inclusion}.  Using the corolla inclusion, each element $x \in X(n,m)$ will also be regarded as an element in $FX(n,m)$.

Each element in $FX$ can be written (not-necessarily uniquely) as a finite sum
\begin{equation}
\label{fxelement}
\sum \pm y_{i_1}\cdots y_{i_k},
\end{equation}
in which each $y_i$ is in $X$, the unit element $\unit \in FX$, or a non-identity permutation.  Each monomial $y_{i_1}\cdots y_{i_k}$ is parenthesized such that each product is either the horizontal composition or the vertical composition in $FX$.

%%%%%%%%%%%%%%%%%%%%%%%%%%%%%%%%%%%%%%%%%%%%%%%%%
\subsection{PROPs via generators and relations}
\label{genrel}

Let $X$ be a $\Sigma$-bimodule and $R$ be a subset of the free PROP $FX$.  Then the quotient PROP
\begin{equation}
\label{fxj}
\frac{FX}{\langle R \rangle}
\end{equation}
is said to have \emph{generators} $X$ and \emph{generating relations} $R$.  Here $\langle R \rangle$ is the ideal in $FX$ generated by $R$.  For an element $x \in X$, its image under the composition
\[
X \xrightarrow{\text{natural}} FX \xrightarrow{\text{quotient}} FX/\langle R \rangle
\]
is also denoted by $x$.

Let $\sP$ be the PROP $FX/\langle R \rangle$ and $V$ be a $\sP$-algebra with structure map $\lambda \colon \sP \to \sE_V$ (Definition \ref{def:propalgebra}).  The structure map $\lambda$ induces a map
\[
FX \xrightarrow{\text{quotient}} \sP \xrightarrow{\lambda} \sE_V
\]
of PROPs, which is also denoted by $\lambda$, whose restriction to $R$ is $0$.  By the free-forgetful adjunction, the map $\lambda \colon FX \to \sE_V$ of PROPs is determined by the restriction
\[
\lambda \colon X \to \sE_V,
\]
which is a map of $\Sigma$-bimodules.  Using the Hom-tensor adjunction, we have generating operations
\[
\lambda(x) \colon V^{\otimes m} \to V^{\otimes n}
\]
for $x \in X(n,m)$.  Moreover, if $\sum \pm y_{i_1}\cdots y_{i_k} \in R$ as in \eqref{fxelement}, then we have the generating relation
\[
\sum \pm \lambda(y_{i_1}) \cdots \lambda(y_{i_k}) = 0
\]
in the endomorphism PROP $\sE_V$.
%where $\gamma_G$ is the structure map \eqref{gammag} in the endomorphism PROP $\sE_V$.  In what follows, we often abbreviate the operation $\lambda(x)$ to $x$ and the element $\sum \pm \lambda(y_{i_1}) \cdots \lambda(y_{i_k}) \in \sE_V(G)$ to $\sum \pm y_{i_1}\cdots y_{i_k}$.

%%%%%%%%%%%%%%%%%%%%%%%%%%%%%%%
\section{Hom-algebras of PROPs}
\label{sec:homalgebra}
%%%%%%%%%%%%%%%%%%%%%%%%%%%%%%%

In this section, we define (multiplicative) Hom-algebras over a PROP and establish some of their basic properties.

%%%%%%%%%%%%%%%%%%%%%%%
\subsection{Setting}
\label{propsetting}

Let $X$ be a $\Sigma$-bimodule and $R$ be a subset of the free PROP $FX$.  Throughout this section,
\[
\sP = \frac{FX}{\langle R \rangle}
\]
denotes the PROP with generators $X$ and generating relations $R$ as in \eqref{fxj}.  Assume that for each element in $R$, a specific representation as a finite sum
\[
\sum \pm y_{i_1}\cdots y_{i_k}
\]
as in \eqref{fxelement} is already chosen.

%%%%%%%%%%%%%%%%%%%%%%%%%%%%%%%%%%%%%%%%%%%%%%%%%%%
\subsection{PROPs for multiplicative Hom-algebras}
\label{mhomhalg}

Define the $\Sigma$-bimodule
\[
X_{mh} = X \oplus \bk\langle \alpha \rangle.
\]
The subscript $mh$ refers to \emph{multiplicative Hom}.  The elements $\alpha$ is in $X_{mh}(1,1)$.  The symbol $\bk \langle \alpha \rangle$ denotes the $\Sigma$-bimodule generated by the element $\alpha$ of biarity $(1,1)$.

Define the subset $R_{mh} \subseteq FX_{mh}$ as consisting of the following two types of elements:
\begin{enumerate}
\item
For each $x \in X(n,m)$, we have
\[
x \circ \alpha^{\otimes m} - \alpha^{\otimes n} \circ x \in R_{mh},
\]
where $x$ and $\alpha$ are regarded as elements in $FX_{mh}$ via the corolla inclusion \eqref{corinclusion}.
\item
If $\sum \pm y_{i_1}\cdots y_{i_k} \in R$, then
\[
\sum \pm \ybar_{i_1}\cdots \ybar_{i_k} \in R_{mh},
\]
where
\[
\ybar_i =
\begin{cases}
y_i & \text{if $y_i \not= \unit$},\\
\alpha & \text{if $y_i = \unit$}.
\end{cases}
\]
\end{enumerate}

\begin{definition}
\label{def:mhomalg}
Algebras over the PROP
\begin{equation}
\label{pmh}
\sP_{mh} = \frac{FX_{mh}}{\langle R_{mh}\rangle}
\end{equation}
are called \emph{multiplicative Hom-$\sP$-algebras}.  A morphism of multiplicative Hom-$\sP$-algebras is a morphism of $\sP_{mh}$-algebras.
\end{definition}

By the discussion in section \ref{genrel}, a multiplicative Hom-$\sP$-algebra $\lambda \colon \sP_{mh} \to \sE_A$ has generating operations:
\[
\begin{split}
\lambda(x) & \colon A^{\otimes m} \to A^{\otimes n} \quad\text{for $x \in X(n,m)$ and}\\
\lambda(\alpha) & \colon A \to A \quad \text{(twisting map)}.
\end{split}
\]
These generating operations satisfy generating relations as specified by $R_{mh}$.  In particular:
\begin{enumerate}
\item
The element $(x \circ \alpha^{\otimes m} - \alpha^{\otimes n} \circ x) \in R_{mh}$ says that the twisting map $\lambda(\alpha)$ is compatible with each generating operation $\lambda(x) \colon A^{\otimes m} \to A^{\otimes n}$ for $x \in X(n,m)$, in the sense that the square
\begin{equation}
\label{alphacompat}
\nicearrow
\xymatrix{
A^{\otimes m} \ar[r]^-{\lambda(x)} \ar[d]_{\lambda(\alpha)^{\otimes m}} & A^{\otimes n} \ar[d]^{\lambda(\alpha)^{\otimes n}}\\
A^{\otimes m} \ar[r]^-{\lambda(x)} & A^{\otimes n}
}
\end{equation}
commutes.
\item
The element $\sum \pm \ybar_{i_1}\cdots \ybar_{i_k} \in R_{mh}$ says that the generating operations satisfy generating relations that are obtained from those of $\sP$-algebras by replacing each occurrence of the identity map by the twisting map $\lambda(\alpha)$.
\end{enumerate}

%%%%%%%%%%%%%%%%%%%%%%%%%%%%%%%%%%%%%
\subsection{PROPs for Hom-algebras}
\label{prophomalg}

Now we discuss Hom-algebras that are less restrictive than the multiplicative ones by allowing more twisting maps and removing the multiplicativity condition.

Let $I$ be the set of unit elements $\unit$ that appear in the elements $\sum \pm y_{i_1}\cdots y_{i_k} \in R$ (as in \eqref{fxelement}).  We assume that $I$ is non-empty.  If $i \in I$, we sometimes write $\unit_i$ for the corresponding unit element appearing in $R$.  Let $S \subseteq I$ be a non-empty subset and $\theta$ be a partition of $S$ into non-empty subsets.  \emph{The default assumption is that $S = I$, unless otherwise specified.}

Only the unit elements in $S$ are to be replaced by twisting maps, and there is one twisting map for each block in the partition $\theta$.  More precisely, define the $\Sigma$-bimodule
\[
X_{h, \theta} = X \oplus \bk \langle \{\alpha_p\}_{p \in \theta}\rangle,
\]
with each $\alpha_p \in X_h(1,1)$.  The symbol $\bk \langle \{\alpha_p\}_{p \in \theta}\rangle$ denotes the $\Sigma$-bimodule generated by the elements $\alpha_p$ for $p \in \theta$, all having biarity $(1,1)$.

Define the subset $R_{h, \theta} \subseteq FX_{h, \theta}$ as consisting of the following elements: If $\sum \pm y_{i_1}\cdots y_{i_k} \in R$, then
\[
\sum \pm y_{i_1}'\cdots y_{i_k}' \in R_{h, \theta},
\]
where
\[
y_i' =
\begin{cases}
y_i & \text{if either $y_i \not= \unit$ or $y_i = \unit \not\in S$},\\
\alpha_p & \text{if $y_i = \unit \in p \in \theta$}.
\end{cases}
\]

\begin{definition}
\label{def:homalg}
Algebras over the PROP
\begin{equation}
\label{ph}
\sP_{h, \theta} = \frac{FX_{h, \theta}}{\langle R_{h, \theta}\rangle}
\end{equation}
are called \emph{Hom-$\sP$-algebras of type $\theta$}.  A morphism of Hom-$\sP$-algebras of type $\theta$ is a morphism of $\sP_{h,\theta}$-algebras.
\end{definition}

%If $\theta$ is understood from the context, then we will not mention it.

By the discussion in section \ref{genrel}, a Hom-$\sP$-algebra of type $\theta$, $\lambda \colon \sP_{h,\theta} \to \sE_A$, has generating operations:
\[
\begin{split}
\lambda(x) & \colon A^{\otimes m} \to A^{\otimes n} \quad\text{for $x \in X(n,m)$ and}\\
\lambda(\alpha_p) & \colon A \to A \quad \text{for $p \in \theta$ (twisting maps)}.
\end{split}
\]
These generating operations satisfy generating relations as specified by $R_{h, \theta}$.  In particular, generating relations in $A$ are obtained from those of $\sP$-algebras by replacing each identity map in $S$ by a twisting map $\lambda(\alpha_p)$.

\begin{example}
There is a partition $\theta_{max}$ with $S = I$ that consists of one-element subsets in $I$.  In this case, there are as many twisting maps as there are $\unit$ in $R$.
\end{example}

\begin{example}
There is the trivial partition
\begin{equation}
\label{thetamin}
\theta_{min} = \{S\}.
\end{equation}
In this case, there is exactly one twisting map.  In particular, a multiplicative Hom-$\sP$-algebra is exactly a Hom-$\sP$-algebra of type $\theta_{min}$ with $S = I$ in which the twisting map is compatible with the generating operations $\lambda(x) \colon A^{\otimes m} \to A^{\otimes n}$ for all $x \in X(n,m)$.
\end{example}

The following result says, among other things, that $\sP$-algebras can also be regarded as Hom-$\sP$-algebras with identity twisting maps.

\begin{proposition}
\label{hompalg}
Let $\sP = F(X)/\langle R \rangle$ be a PROP, $I$ be the set of unit elements $\unit$ in $R$, $S \subseteq I$ be a non-empty subset, and $\theta$ be a partition of $S$ into non-empty subsets.  Then there are solid-arrow maps of PROPs
\begin{equation}
\label{ppi}
\nicearrow
\xymatrix{
& \sP_{h,\theta} \ar[d]^{\pi} \ar@{.>}[dl]_{\pi_1} \\
\sP_{mh} \ar[r]_{\pi_2} & \sP
}
\end{equation}
determined by
\[
\begin{split}
\pi|_X &= Id,\quad \pi(\alpha_p) = \unit \quad \text{for all $p \in \theta$},\\
\pi_2|_X &= Id,\quad \pi_2(\alpha) = \unit.
\end{split}
\]
If, in addition, $S = I$ then the map $\pi_1$ of PROPs exists such that the triangle \eqref{ppi} commutes, and $\pi_1$ is determined by
\[
\pi_1|_X = Id, \quad \pi_1(\alpha_p) = \alpha \quad \text{for all $p \in \theta$}.
\]
\end{proposition}

\begin{proof}
Consider the map of $\Sigma$-bimodules
\[
\eta \colon X_{h,\theta} = X \oplus \bk \langle \{\alpha_p\}_{p \in \theta}\rangle \to FX
\]
determined by
\[
\eta(x) = x \quad\text{for all $x \in X$},\andspace
\eta(\alpha_p) = \unit \quad\text{for all $p \in \theta$}.
\]
The adjoint of $\eta$ is a map of PROPs
\[
\eta' \colon FX_{h,\theta} \to FX.
\]
Consider the solid-arrow diagram of PROPs
\[
\nicearrow
\xymatrix{
FX_{h,\theta} \ar[r]^-{\eta'} \ar[d]_{q_1} & FX \ar[d]^{q_2}\\
\sP_{h,\theta} \ar@{.>}[r]^-{\pi} & \sP,
}
\]
in which each $q_i$ denotes the quotient map.  The image of $R_{h,\theta}$ under $\eta'$ is exactly $R$, which is annihilated by $q_2$.  Therefore, the composition $q_2 \circ \eta'$ factors through $q_1$, yielding the dotted-arrow $\pi$ with the desired properties.

The assertions regarding $\pi_1$ and $\pi_2$ are proved by essentially the same arguments as in  the previous paragraph.
\end{proof}

Proposition \ref{hompalg} can be interpreted as follows.
\begin{enumerate}
\item
Suppose $\lambda \colon \sP \to \sE_V$ is a $\sP$-algebra (Definition \ref{def:propalgebra}). Then for an arbitrary partition $\theta$ of $S$, the composition
\[
\sP_{h,\theta} \xrightarrow{\pi} \sP \xrightarrow{\lambda} \sE_V
\]
gives $V$ the structure of a Hom-$\sP$-algebra of type $\theta$, in which the twisting maps $\alpha_p$ for $p \in \theta$ are all equal to the identity map on $V$.
\item
Likewise, the composition
\[
\sP_{mh} \xrightarrow{\pi_2} \sP \xrightarrow{\lambda} \sE_V
\]
gives $V$ the structure of a multiplicative Hom-$\sP$-algebra, in which the twisting map $\alpha$ is the identity map on $V$.
\item
Suppose $S = I$, $\theta$ is a partition of $I$, and $\rho \colon \sP_{mh} \to \sE_V$ is a multiplicative Hom-$\sP$-algebra.  Then the composition
\[
\sP_{h,\theta} \xrightarrow{\pi_1} \sP_{mh} \xrightarrow{\rho} \sE_V
\]
gives $V$ the structure of a Hom-$\sP$-algebra of type $\theta$, in which the twisting maps $\alpha_p$ for $p \in \theta$ are all equal to each other.
\end{enumerate}

%%%%%%%%%%%%%%%%%%%%%%%%%%%%%%%%%%%
\section{Examples of Hom-algebras}
\label{sec:examples}
%%%%%%%%%%%%%%%%%%%%%%%%%%%%%%%%%%%

In this section, we present some examples of Hom-algebras and show how these Hom-algebras arise in our general PROP setting.

%%%%%%%%%%%%%%%%%%%%%%%%%%%%%%%%%%%%%%%%%
\subsection{$G$-Hom-associative algebras}
\label{ex:ghomas}

Let $G$ be a subgroup of the symmetric group $\Sigma_3$.  Let
\[
as(x,y,z) = (xy)z - x(yz)
\]
denote the associator.  Then a \emph{$G$-associative algebra} \cite{gr} is an algebra $(A,\mu)$ with multiplication $\mu \colon A^{\otimes 2} \to A$ that satisfies
\[
\sum_{\sigma \in G} \signsigma as \circ \sigma = 0.
\]
In particular:
\begin{enumerate}
\item
An $\{e\}$-associative algebra is exactly an associative algebra.
\item
An $\{Id, (1~2)\}$-associative algebra is exactly a left pre-Lie algebra.
\item
An $\{Id, (2~3)\}$-associative algebra is exactly a right pre-Lie algebra.
\item
An $A_3$-associative algebra with anti-symmetric multiplication is exactly a Lie algebra.
\item
A $\Sigma_3$-associative algebra is exactly a Lie-admissible algebra.
\end{enumerate}

The Hom-type analogue of a $G$-associative algebra is defined as follows.  Define the \emph{Hom-associator} as
\[
as_\alpha(x,y,z) = (xy)\alpha(z) - \alpha(x)(yz) = \left(\mu \circ (\mu \otimes \alpha - \alpha \otimes \mu)\right)(x,y,z).
\]
A \emph{$G$-Hom-associative algebra} \cite{ms1,yau2} is a triple $(A,\mu,\alpha)$, in which $A$ is a $\bk$-module, $\mu \colon A^{\otimes 2} \to A$, and $\alpha \colon A \to A$, that satisfies
\begin{equation}
\label{ghomass}
\sum_{\sigma \in G} \signsigma as_\alpha \circ \sigma = 0.
\end{equation}
In particular:
\begin{enumerate}
\item
A \emph{Hom-associative algebra} is defined as an $\{e\}$-Hom-associative algebra.  The axiom \eqref{ghomass} in this case is
\begin{equation}
\label{homassociativity}
(xy)\alpha(z) = \alpha(x)(yz),
\end{equation}
which is called \emph{Hom-associativity}.
\item
A \emph{Hom-Lie algebra} \cite{hls,ms1} is equivalent to an $A_3$-Hom-associative algebra with an anti-symmetric multiplication $\mu$.  With $\mu$ anti-symmetric, the axiom \eqref{ghomass} in this case is equivalent to
\[
(xy)\alpha(z) + (zx)\alpha(y) + (yz)\alpha(x) = 0,
\]
which is called the \emph{Hom-Jacobi identity}.
\item
A \emph{Hom-Lie admissible algebra} \cite{ms1,yau2} is equivalent to a $\Sigma_3$-Hom-associative algebra.
\end{enumerate}

The PROP (actually an operad suffices) for $G$-associative algebra is
\begin{equation}
\label{asgprop}
\As(G) = \frac{F(\mu)}{\langle R \rangle},
\end{equation}
where $F(\mu)$ is the free PROP generated by an element $\mu$ of biarity $(1,2)$, and $R$ consists of the single element
\[
\sum_{\sigma \in G} \signsigma \left\{\mu \circ (\mu \otimes \unit) \circ \sigma - \mu \circ (\unit \otimes \mu) \circ \sigma\right\} \in F(\mu).
\]
The unit element $\unit$ appears $2|G|$ times in $R$, so the set $I$ consists of $2|G|$ copies of $\unit$.

A $G$-Hom-associative algebra is precisely a Hom-$\As(G)$-algebra of type $\theta_{min}$ (Definition \ref{def:homalg}), where $\theta_{min}$ is the trivial partition \eqref{thetamin} of $I$.  Indeed, a Hom-$\As(G)$-algebra of type $\theta_{min}$ has generating operations
\[
\mu \colon A^{\otimes 2} \to A \andspace \alpha \colon A \to A,
\]
satisfying the generating relation
\[
\sum_{\sigma \in G} \signsigma \left\{\mu \circ (\mu \otimes \alpha) - \mu \circ (\alpha \otimes \mu)\right\} \circ \sigma = 0.
\]
This relation is exactly the relation \eqref{ghomass} that defines $G$-Hom-associative algebra.

%%%%%%%%%%%%%%%%%%%%%%%%%%%%%%%%%%%%%%%%%%%%%%%%%%%%
\subsection{Variants of Hom-associative algebras}
\label{ex:otherhomas}

As discussed above, a Hom-associative algebra $(A,\mu,\alpha)$ is defined by the Hom-associativity condition \eqref{homassociativity}.  Variations of Hom-associativity have also been studied \cite{fg}.  Here we show how these other types of Hom-associative algebras arise in our general PROP setting.

For example, a \emph{Hom-associative algebra of type $II_1$} \cite{fg} $(A,\mu,\alpha)$ is defined by the axiom
\begin{equation}
\label{II1}
\left(\alpha(x) \alpha(y)\right)z = x \left(\alpha(y) \alpha(z)\right).
\end{equation}
Consider the PROP (actually an operad suffices) for associative algebras
\begin{equation}
\label{asprop}
\As = \frac{F(\mu)}{\langle R \rangle}.
\end{equation}
Here $F(\mu)$ is the free PROP generated by an element $\mu$ of biarity $(1,2)$, and $R$ consists of the single element
\[
as_{II_1} = \mu \circ (\mu \otimes \unit_1) \circ (\unit_2 \otimes \unit_3 \otimes \unit_4) - \mu \circ (\unit_5 \otimes \mu) \circ (\unit_6 \otimes \unit_7 \otimes \unit_8) \in F(\mu)
\]
in which each $\unit_i$ is a copy of the unit element $\unit \in F(\mu)$.  In this case, we have $I = \{\unit_i\}_{i=1}^8$.

We take the subset
\[
S = \{\unit_2, \unit_3, \unit_7, \unit_8\} \subseteq I
\]
and the trivial partition $\theta_{min}$ \eqref{thetamin} of $S$.  Then Hom-$\As$-algebras of type $\theta_{min}$ (Definition \ref{def:homalg}) are exactly the Hom-associative algebras of type $II_1$.  Indeed, a Hom-$\As$-algebras of type $\theta_{min}$ has generating operations
\[
\mu \colon A^{\otimes 2} \to A \andspace \alpha \colon A \to A,
\]
satisfying the generating relation
\[
\mu \circ (\mu \otimes Id) \circ (\alpha \otimes \alpha \otimes Id) - \mu \circ (Id \otimes \mu) \circ (Id \otimes \alpha \otimes \alpha) = 0.
\]
This relation is exactly the relation \eqref{II1} that defines Hom-associative algebras of type $II_1$.

All other types of Hom-associative algebras in \cite{fg} can be obtained similarly.  For examples, a \emph{Hom-associative algebra of type III} $(A,\mu,\alpha)$ is defined by the axiom
\[
\alpha\left((xy)z\right) = \alpha\left(x(yz)\right).
\]
Let $\As$ be as in \eqref{asprop}, except that $R$ consists of the single element
\[
as_{III} = \unit_1 \circ \mu \circ (\mu \otimes \unit_2) - \unit_3 \circ \mu \circ (\unit_4 \otimes \mu) \in F(\mu).
\]
In this case, we have $I = \{\unit_i\}_{i=1}^4$, and we take
\[
S = \{\unit_1, \unit_3\}.
\]
Then Hom-$\As$-algebras of type $\theta_{min}$ are exactly the Hom-associative algebras of type III.

%%%%%%%%%%%%%%%%%%%%%%%%%%%%%%%%
\subsection{Hom-Nambu algebras}
\label{ex:nambu}

For elements $x_l$ in a $\bk$-module $V$, we use the abbreviation
\begin{equation}
\label{xij}
x_{i,j} = (x_i, x_{i+1}, \ldots, x_j)
\end{equation}
whenever $i \leq j$.  If $i > j$, then $x_{i,j}$ denotes the empty sequence.  Likewise, if $f_k \colon V \to W$ are maps, then we use the abbreviation
\[
f_{k,l}(x_{i,j}) = (f_k(x_i), f_{k+1}(x_{i+1}), \ldots , f_l(x_j)),
\]
provided $l-k = j-i \geq 0$; otherwise, $f_{k,l}(x_{i,j})$ denotes the empty sequence.

Fix an integer $n \geq 2$.  An \emph{$n$-ary Nambu-algebra} \cite{filippov,nambu} is a pair $(V,\bracket)$ in which $V$ is a $\bk$-module and $\bracket \colon V^{\otimes n} \to V$ is an $n$-ary operation such that the \emph{$n$-ary Nambu identity}
\begin{equation}
\label{nambuid}
\left[x_{1,n-1}, [y_{1,n}]\right] = \sum_{i=1}^n \left[y_{1,i-1}, [x_{1,n-1}, y_i], y_{i+1,n}\right]
\end{equation}
holds.  The $n$-ary Nambu identity is an $n$-ary generalization of the Jacobi identity, in the sense that it says that
\[
[x_{1,n-1}, -] \colon V \to V
\]
is a derivation with respect to the bracket $\bracket$.

The Hom-type analogue of an $n$-ary Nambu algebra was introduced in \cite{ams3} and studied further in \cite{ams1,ams2,yau7,yau8,yau9}.  An \emph{$n$-ary Hom-Nambu algebra} is a triple $(V,\bracket,\alpha = (\alpha_{1,n-1}))$ in which:
\begin{enumerate}
\item
$V$ is a $\bk$-module, 
\item
$\bracket \colon V^{\otimes n} \to V$ is an $n$-ary operation, and 
\item
each $\alpha_i \colon V \to V$ is a linear map, 
\end{enumerate}
such that the \emph{$n$-ary Hom-Nambu identity}
\begin{equation}
\label{nnambuid}
\left[\alpha_{1,n-1}(x_{1,n-1}), [y_{1,n}]\right]
= \sum_{i=1}^n \left[\alpha_{1,i-1}(y_{1,i-1}), [x_{1,n-1},y_i], \alpha_{i,n-1}(y_{i+1,n})\right]
\end{equation}
holds.

To describe $n$-ary Hom-Nambu algebras in the general PROP context, define
\[
M_i = (i-1)(n-1)
\]
for $i = 1,2,\ldots$.  The PROP (actually an operad suffices) for $n$-ary Nambu algebras is
\begin{equation}
\label{namprop}
\Nam = \frac{F(\mu)}{\langle R \rangle}.
\end{equation}
Here $F(\mu)$ is the free PROP on a generator $\mu$ of biarity $(1,n)$.  The set $R$ consists of the single element
\[
\begin{split}
r &= \mu \circ \left(\unit_{1+M_{n+1},n-1+M_{n+1}} \otimes \mu \right)\\
&- \sum_{i=1}^n \mu \circ \left(\unit_{1+M_i, i-1+M_i} \otimes \mu \otimes \unit_{i+M_i, n-1+M_i}\right) \circ \sigma_i \in F(\mu),
\end{split}
\]
where each $\unit_j$ is a copy of the unit element in $F(\mu)$ and $\sigma_i \in \Sigma_{2n-1}$ is the block permutation
\[
\sigma_i\left(x_{1,n-1}, y_{1,n}\right) = \left(y_{1,i-1}, x_{1,n-1}, y_{i,n}\right).
\]
There are $M_{n+2} = (n+1)(n-1)$ copies of $\unit$ in $R$, so $I = \{\unit_j\}_{j=1}^{(n+1)(n-1)}$.

With $S = I$ we take the partition
\[
\theta = p_1 \sqcup \cdots \sqcup p_{n-1}
\]
of $I$, where
\[
p_j = \left\{\unit_{j+M_1}, \ldots \unit_{j+M_{n+1}}\right\}
\]
for $1 \leq j \leq n-1$.  Then Hom-$\Nam$-algebras of type $\theta$  (Definition \ref{def:homalg}) are exactly the $n$-ary Hom-Nambu algebras.  All other $n$-ary Hom-algebras considered in \cite{ams3} can be obtained similarly in the general PROP context.

%%%%%%%%%%%%%%%%%%%%%%%%%%%%
\subsection{Hom-bialgebras}
\label{ex:hombi}

Two Hom-type generalizations of bialgebras have been studied \cite{ms4,yau4}.  Recall that a \emph{bialgebra} $(A,\mu,\Delta)$ is a $\bk$-module $A$ with an associative multiplication $\mu \colon A^{\otimes 2} \to A$ and a coassociative comultiplication $\Delta \colon A \to A^{\otimes 2}$, satisfying the compatibility condition
\begin{equation}
\label{bialgcompat}
\Delta \circ \mu = \mu^{\otimes 2} \circ (2 ~ 3) \circ \Delta^{\otimes 2}.
\end{equation}
Here $(2~3) \in \Sigma_4$ is the permutation that switches the middle two entries.

A \emph{Hom-coassociative coalgebra} \cite{ms4} is a triple $(C,\Delta,\alpha)$, in which $C$ is a $\bk$-module and $\Delta \colon C \to C^{\otimes 2}$ and $\alpha \colon C \to C$ are linear maps, such that the \emph{Hom-coassociativity} condition
\[
(\Delta \otimes \alpha) \circ \Delta = (\alpha \otimes \Delta) \circ \Delta
\]
holds.  The last condition is the dual of Hom-associativity \eqref{homassociativity}.

One Hom-type generalization of a bialgebra has one twisting map.  A \emph{Hom-bialgebra} \cite{yau4} is a quadruple $(A,\mu,\Delta,\alpha)$ in which:
\begin{enumerate}
\item
$(A,\mu,\alpha)$ is a Hom-associativity algebra \eqref{homassociativity}.
\item
$(A,\Delta,\alpha)$ is a Hom-coassociative coalgebra.
\item
The compatibility condition \eqref{bialgcompat} holds.
\end{enumerate}

The other Hom-type generalization of a bialgebra has two twisting maps.  A \emph{generalized Hom-bialgebra} \cite{ms4} is a quintuple $(A,\mu,\Delta,\alpha_1,\alpha_2)$ in which:
\begin{enumerate}
\item
$(A,\mu,\alpha_1)$ is a Hom-associativity algebra \eqref{homassociativity}.
\item
$(A,\Delta,\alpha_2)$ is a Hom-coassociative coalgebra.
\item
The compatibility condition \eqref{bialgcompat} holds.
\end{enumerate}

To put (generalized) Hom-bialgebras in the PROP context, let us first consider the PROP for bialgebras:
\begin{equation}
\label{biprop}
\Bi = \frac{F(\mu,\Delta)}{\langle R \rangle}.
\end{equation}
Here $F(\mu,\Delta)$ is the free PROP on two generators $\mu$ and $\Delta$ of biarities $(1,2)$ and $(2,1)$, respectively.  The set
\[
R = \{as, coas, comp\}
\]
consists of the following three elements in $F(\mu,\Delta)$:
\[
\begin{split}
as &= \mu \circ (\mu \otimes \unit_1) - \mu \circ (\unit_2 \otimes \mu),\\
coas &= (\Delta \otimes \unit_3) \circ \Delta - (\unit_4 \otimes \Delta) \circ \Delta,\\
comp &= \Delta \circ \mu - (\mu \otimes \mu) \circ (2~3) \circ (\Delta \otimes \Delta).
\end{split}
\]
The unit element $\unit \in F(\mu,\Delta)$ appears four times in $R$, so $I = \{\unit_i\}_{i=1}^4$.  We take $S = I$.  Then:
\begin{enumerate}
\item
With the trivial partition $\theta_{min}$ \eqref{thetamin} of $I$, Hom-$\Bi$-algebras of type $\theta_{min}$ (Definition \ref{def:homalg}) are exactly the Hom-bialgebras as defined in \cite{yau4}.
\item
On the other hand, consider the partition
\[
\theta = p_1 \sqcup p_2
\]
of $I$ with
\[
p_1 = \{\unit_1, \unit_2\} \andspace p_2 = \{\unit_3, \unit_4\}.
\]
Then Hom-$\Bi$-algebras of type $\theta$ are exactly the generalized Hom-bialgebras as defined in \cite{ms4}.
\end{enumerate}

%%%%%%%%%%%%%%%%%%%%%%%%%%%%%%%%%%%%%%%%%%
\subsection{The Hom-Yang-Baxter equation}
\label{ex:hybe}

Let $V$ be a $\bk$-module.  A linear map $B \colon V^{\otimes 2} \to V^{\otimes 2}$ is said to be a solution of the \emph{Yang-Baxter equation} (YBE) if it satisfies
\[
(Id_V \otimes B) \circ (B \otimes Id_V) \circ (Id_V \otimes B) = (B \otimes Id_V) \circ (Id_V \otimes B) \circ (B \otimes Id_V).
\]
If, in addition, $B$ is invertible, then it is called an $R$-matrix.  The YBE was first introduced in statistical mechanics \cite{baxter,baxter2,yang}.  It plays a major role in many subjects, including braid group representations, quantum groups, quantum integrable systems, braided categories, and knot invariants.

The Hom-type generalization of the YBE was introduced and studied in \cite{yau3,yau6}.  Let $\alpha \colon V \to V$ be a linear map.  A linear map $B \colon V^{\otimes 2} \to V^{\otimes 2}$ is said to be a solution of the \emph{Hom-Yang-Baxter equation} (HYBE) for $(V,\alpha)$ if it satisfies
\[
(\alpha \otimes B) \circ (B \otimes \alpha) \circ (\alpha \otimes B) = (B \otimes \alpha) \circ (\alpha \otimes B) \circ (B \otimes \alpha)
\]
and
\[
\alpha^{\otimes 2} \circ B = B \circ \alpha^{\otimes 2}.
\]
Now we describe solutions of the HYBE in our general PROP context.

Consider the PROP
\begin{equation}
\label{ybeprop}
\YBE = \frac{F(B)}{\langle R \rangle},
\end{equation}
where $F(B)$ is the free PROP on a generator $B$ of biarity $(2,2)$, and $R$ consists of the single element
\[
\begin{split}
r &= (\unit \otimes B) \circ (B \otimes \unit) \circ (\unit \otimes B)\\
&- (B \otimes \unit) \circ (\unit \otimes B) \circ (B \otimes \unit) \in F(B).
\end{split}
\]
A $\YBE$-algebra is exactly a pair $(V,B)$ in which $V$ is a $\bk$-module and $B \colon  V^{\otimes 2} \to V^{\otimes 2}$ is a solution of the YBE.  A multiplicative Hom-$\YBE$-algebra (Definition \ref{def:mhomalg}) is exactly a triple $(V,\alpha,B)$ in which $V$ is a $\bk$-module and $B$ is a solution of the HYBE for $(V,\alpha)$.

%%%%%%%%%%%%%%%%%%%%%%%%%%%%%%%%%%%%%
\subsection{Hom-$A_\infty$-algebras}
\label{ex:homainf}

Here we work over $\bZ$-graded $\bk$-modules.  An \emph{$A_\infty$-algebra} \cite{stasheff} is a tuple $\left(V,\{m_k\}_{k=1}^\infty\right)$ in which $V$ is a $\bZ$-graded $\bk$-module and $m_k \colon V^{\otimes k} \to V$ is a $k$-ary operation of degree $2-k$, such that the identity
\[
\sum_{l=0}^{n-1} \sum_{k=1}^{n-l} (-1)^\epsilon \, m_{n-k+1} \left(x_{1,l} \otimes m_k(x_{l+1,l+k}) \otimes x_{l+k+1,n}\right) = 0
\]
holds for each $n \geq 1$, in which
\[
\epsilon = (k+1)(l+1)-1 + k(n+|x_1| + \cdots + |x_l|).
\]
Here we are using the abbreviation in \eqref{xij}.

We use the $n$-ary Hom-Nambu identity \eqref{nnambuid} as a guide to define the Hom-type generalization of an $A_\infty$-algebra.  First, the PROP (actually an operad suffices) for $A_\infty$-algebras is
\begin{equation}
\label{ainftyprop}
\ainfty = \frac{F\left(\{\mu_k\}_{k=1}^\infty\right)}{\langle R \rangle},
\end{equation}
in which $F\left(\{\mu_k\}_{k=1}^\infty\right)$ is the free PROP on the generators $\mu_k$ of homogeneous degree $2-k$ and biarity $(1,k)$.  The set $R$ consists of the elements
\begin{equation}
\label{ainftyr}
\sum_{l=0}^{n-1} \sum_{k=1}^{n-l} (-1)^\epsilon \, \mu_{n-k+1} \circ \left(\unit_1 \otimes \cdots \otimes \unit_l \otimes \mu_k \otimes \unit_{l+k+1} \otimes \cdots \otimes \unit_{n}\right)
\end{equation}
for all $n \geq 1$.  There are countably infinitely many copies of the unit element $\unit$ in $R$.

With $S = I$ we take the partition
\[
\theta = \bigsqcup_{i=1}^\infty \theta_i
\]
of $I$, where $\theta_i$ consists of the unit elements labeled $\unit_i$ in \eqref{ainftyr}.  Define a \emph{Hom-$A_\infty$-algebra} as a Hom-$\ainfty$-algebra of type $\theta$ (Definition \ref{def:homalg}).  More explicitly, a Hom-$A_\infty$-algebra is a tuple $\left(V, \{m_k\}_{k=1}^\infty, \{\alpha_k\}_{k=1}^\infty\right)$ in which:
\begin{enumerate}
\item
$V$ is a $\bZ$-graded $\bk$-module,
\item
each $m_k \colon V^{\otimes k} \to V$ is a $k$-ary operation of degree $2-k$, and
\item
each $\alpha_k \colon V \to V$ is a degree $0$ linear map,
\end{enumerate}
such that
\[
\sum_{l=0}^{n-1} \sum_{k=1}^{n-l} (-1)^\epsilon \, m_{n-k+1} \left(\alpha_{1,l}(x_{1,l}) \otimes m_k(x_{l+1,l+k}) \otimes \alpha_{l+k+1,n}(x_{l+k+1,n})\right) = 0
\]
holds for each $n \geq 1$.

%%%%%%%%%%%%%%%%%%%%%%%%%%%%%%%%%%%%%
\subsection{Hom-$L_\infty$-algebras}
\label{ex:homlinf}

We continue to work over $\bZ$-graded $\bk$-modules.  An \emph{$L_\infty$-algebra} \cite{lm} is a tuple $(V,\{l_k\}_{k=1}^\infty)$ in which $V$ is a $\bk$-module and $l_k \colon V^{\otimes k} \to V$ is an anti-symmetric operation of degree $2-k$, such that the condition
\begin{equation}
\label{eq:ln}
\sum_{i+j=n+1} \sum_{\sigma \in \cS_{i,n-i}} \chi(\sigma) (-1)^{i(j-1)} l_j\left(l_i(x_{\sigma(1)}, \ldots , x_{\sigma(i)}), x_{\sigma(i+1)}, \ldots , x_{\sigma(n)}\right) = 0
\end{equation}
holds for each $n \geq 1$.  In \eqref{eq:ln} $\cS_{i,n-i}$ denotes the set of all the $(i,n-i)$-unshuffles for $i \geq 1$, and
\[
\chi(\sigma) = (-1)^\sigma \cdot \varepsilon(\sigma; x_1, \ldots , x_n),
\]
where $(-1)^\sigma$ is the signature of $\sigma$, and $\varepsilon(\sigma; x_1, \ldots , x_n)$ is the Koszul sign determined by
\[
x_1 \wedge \cdots \wedge x_n = \varepsilon(\sigma; x_1, \ldots , x_n) \cdot x_{\sigma(1)} \wedge \cdots \wedge x_{\sigma(n)}.
\]
The anti-symmetry of $l_k$ means that
\[
\chi(\sigma)l_k(x_{\sigma(1)}, \ldots , x_{\sigma(k)}) = l_k(x_1, \ldots , x_k)
\]
for $\sigma \in \Sigma_k$ and $x_1, \ldots , x_k \in V$.

The PROP (actually an operad suffices) for $L_\infty$-algebras is
\begin{equation}
\label{linftyprop}
\linfty = \frac{F(\{\mu_k\}_{k=1}^\infty)}{\langle R \rangle},
\end{equation}
where $F(\{\mu_k\}_{k=1}^\infty)$ is the free PROP on the generators $\mu_k$ of homogeneous degree $2-k$ and biarity $(1,k)$.  The set $R$ consists of the elements
\[
\mu_k - \chi(\sigma) \mu_k \circ \sigma
\]
for all $k \geq 1$ and $\sigma \in \Sigma_k$, and
\begin{equation}
\label{linftyr}
\sum_{i+j=n+1} \sum_{\sigma \in \cS_{i,n-i}} \chi(\sigma) (-1)^{i(j-1)} \mu_j \circ \left(\mu_i \otimes \unit_1 \otimes \cdots \otimes \unit_{n-i}\right) \circ \sigma
\end{equation}
for all $n \geq 1$.  There are countably infinitely many copies of the unit element $\unit$ in $R$.

With $S = I$ we take the partition
\[
\theta = \bigsqcup_{i=1}^\infty \theta_i
\]
of $I$, where $\theta_i$ consists of the unit elements labeled $\unit_i$ in \eqref{linftyr}.  Define a \emph{Hom-$L_\infty$-algebra} as a Hom-$\linfty$-algebra of type $\theta$ (Definition \ref{def:homalg}).  More explicitly, a Hom-$L_\infty$-algebra is a tuple $\left(V, \{l_k\}_{k=1}^\infty, \{\alpha_k\}_{k=1}^\infty\right)$ in which:
\begin{enumerate}
\item
$V$ is a $\bZ$-graded $\bk$-module,
\item
each $l_k \colon V^{\otimes k} \to V$ is an anti-symmetric $k$-ary operation of degree $2-k$, and
\item
each $\alpha_k \colon V \to V$ is a degree $0$ linear map,
\end{enumerate}
such that
\[
\sum_{i+j=n+1} \sum_{\sigma \in \cS_{i,n-i}} \chi(\sigma) (-1)^{i(j-1)} l_j\left(l_i(x_{\sigma(1)}, \ldots , x_{\sigma(i)}), \alpha_1(x_{\sigma(i+1)}), \ldots , \alpha_{n-i}(x_{\sigma(n)})\right) = 0
\]
holds for each $n \geq 1$.

%%%%%%%%%%%%%%%%%%%%%%%%%%%%%%%%%%%%%%%%%%%%%%%%%%
\section{Twisting constructions for Hom-algebras}
\label{sec:twit}
%%%%%%%%%%%%%%%%%%%%%%%%%%%%%%%%%%%%%%%%%%%%%%%%%%

In this section, we provide general construction results for Hom-algebras over a PROP that generalize and unify many such results in the literature.  We also give a partial classification of Hom-algebras.

Throughout this section, let $\sP = FX/\langle R \rangle$ be a PROP as in section \ref{propsetting}.  Let us first introduce some notations.

\begin{definition}
Let $x_p, y_p$, and $z_p$ for $p \geq 1$ be elements in a PROP $\sP = FX/\langle R \rangle$.
\begin{enumerate}
\item
For $i \leq j$ define
\[
x_{i,j} = \bigotimes_{l=i}^j x_l = x_i \otimes x_{i+1} \otimes \cdots \otimes x_j
\]
as the iterated horizontal composition of $x_i, \ldots , x_j$ in the stated order.  This is well-defined because the horizontal composition is associative.
\item
For $i \leq j$ define
\[
\bigo_{l=i}^j y_l = y_i \circ y_{i+1} \circ \cdots \circ y_j
\]
as the iterated vertical composition of $y_i, \ldots , y_j$ in the stated order, provided each vertical composition makes sense.  This is well-defined because the vertical composition is associative.
\item
For positive integers $n_1, n_2, \ldots$, set
\begin{equation}
\label{nj}
N_j = n_1 + \cdots + n_{j-1}
\end{equation}
with $N_1 = 0$.  For $k \geq 1$ by a \emph{degree $k$ monomial}, we mean an element in $\sP$ of the form
\begin{equation}
\label{monomial}
\sigma_0 \circ \left(\bigo_{j=1}^k x_{N_j+1,N_{j+1}} \circ \sigma_j \right)
\end{equation}
in which each $x_i \in X \sqcup \{\unit\}$ and each $\sigma_j$ is a permutation.
\item
An element in $\sP$ is said to be of \emph{homogeneous degree $k$} if it has the form
\[
\sum \pm z_l
\]
in which the sum is finite and non-empty, and each $z_l$ is a degree $k$ monomial.
\item
A PROP $\sP = FX/\langle R \rangle$ is said to be \emph{normal} if each element in $R$ is of homogeneous degree $k$ for some $k \geq 1$.
\end{enumerate}
\end{definition}

\begin{example}
The PROPs $\As(G)$ \eqref{asgprop}, $\As$ \eqref{asprop}, $\Nam$ \eqref{namprop}, $\Bi$ \eqref{biprop}, $\YBE$ \eqref{ybeprop}, $\ainfty$ \eqref{ainftyprop}, and $\linfty$ \eqref{linftyprop} are all normal.
\end{example}

The following result is our main twisting result for Hom-algebras. It says that over normal PROPs, a Hom-algebra gives rise to another one in which all the generating operations are twisted by an endomorphism.  Recall the notations and conventions from section \ref{prophomalg} and Definition \ref{def:homalg}.

\begin{theorem}
\label{thm:twist}
Let $\sP = FX/\langle R \rangle$ be a normal PROP and $\theta$ be an arbitrary partition of $I$.  Suppose
\[
\lambda \colon \sP_{h,\theta} \to \sE_V
\]
is a Hom-$\sP$-algebra of type $\theta$, and $\beta \colon V \to V$ is a morphism of Hom-$\sP$-algebras of type $\theta$.  Then there is another Hom-$\sP$-algebra of type $\theta$
\begin{equation}
\label{lb}
\lambda_\beta \colon \sP_{h,\theta} \to \sE_V
\end{equation}
determined by
\[
\lambda_\beta(y) = \beta^{\otimes q} \circ \lambda(y)
\]
for all $y \in X_{h,\theta}(q,p)$.
\end{theorem}

\begin{proof}
Consider the map
\[
\xi \colon X_{h,\theta} = X \oplus \bk\langle\{\alpha_p\}_{p \in \theta}\rangle \to \sE_V
\]
of $\Sigma$-bimodules given by
\[
\xi(y) = \beta^{\otimes q} \circ \lambda(y)
\]
for $y \in X_{h,\theta}(q,p)$.  Its adjoint is a map
\[
\xi' \colon FX_{h,\theta} \to \sE_V
\]
of PROPs.  To prove the Theorem, it suffices to show that $\xi'$ factors through the quotient map $FX_{h,\theta} \to \sP_{h,\theta}$, i.e., that
\begin{equation}
\label{xir}
\xi'(R_{h,\theta}) = 0.
\end{equation}

By normality a typical element in $R$ is, using the notations in \eqref{nj} and \eqref{monomial}, a finite sum of degree $k$ monomials
\[
r = \sum \pm \sigma_0 \circ \left(\bigo_{j=1}^k x_{N_j+1,N_{j+1}} \circ \sigma_j \right)
\]
for some $k \geq 1$, in which each $x_i = \unit$ or $x_i \in X$.  Therefore, a typical element in $R_{h,\theta}$ is a finite sum
\begin{equation}
\label{r'}
r' = \sum \pm \sigma_0 \circ \left(\bigo_{j=1}^k x'_{N_j+1,N_{j+1}} \circ \sigma_j \right),
\end{equation}
in which
\[
x_i' =
\begin{cases}
x_i & \text{if $x_i \not= \unit$},\\
\alpha_p & \text{if $x_i = \unit \in p \in \theta$}.
\end{cases}
\]
Since $S = I$ we have that each $x_i' \in X_{h,\theta}$.  Suppose the biarity of $x_i'$ is $(q_i,p_i)$.  It follows that
\begin{equation}
\label{xixi}
\xi'(x_i') = \beta^{\otimes q_i} \circ \lambda(x_i').
\end{equation}
We will show that
\[
\xi'(r') = 0,
\]
which immediately implies \eqref{xir}.

Since $x_i'$ has biarity $(q_i,p_i)$, $x'_{N_j+1,N_{j+1}}$ has biarity $(Q_j,P_j)$, where
\[
P_j = \sum_{i=N_j+1}^{N_{j+1}} p_i \andspace Q_j = \sum_{i=N_j+1}^{N_{j+1}} q_i
\]
From the expression of $r'$ \eqref{r'}, it follows that
\begin{equation}
\label{PQ}
P_j = Q_{j+1}
\end{equation}
for $1 \leq j \leq k-1$.  Since $\beta$ is a morphism of $\sP_{h,\theta}$-algebras, we have
\begin{equation}
\label{betalambda}
\begin{split}
\beta^{\otimes Q_j} \circ \left( \bigotimes_{i=N_j+1}^{N_{j+1}} \lambda(x_i')\right)
&= \beta^{\otimes Q_j} \circ \lambda\left(x'_{N_j+1,N_{j+1}}\right)\\
&= \lambda\left(x'_{N_j+1,N_{j+1}}\right) \circ \beta^{\otimes P_j}\\
&= \left(\bigotimes_{i=N_j+1}^{N_{j+1}} \lambda(x_i')\right) \circ \beta^{\otimes Q_{j+1}} \quad\text{(by \eqref{PQ})}
\end{split}
\end{equation}
for $1 \leq j \leq k-1$.  To simplify the typography below, we will omit the summation sign $\Sigma$ and the permutations $\sigma_j$ in $r'$ because they will not affect the computation at all.  We now compute as follows:
\[
\begin{split}
\xi'(r') &= \bigo_{j=1}^k \left(\bigotimes_{i=N_j+1}^{N_{j+1}} \xi'(x_i')\right)\\
&= \bigo_{j=1}^k \left(\bigotimes_{i=N_j+1}^{N_{j+1}} \left\{ \beta^{\otimes q_i} \circ \lambda(x_i')\right\}\right) \quad\text{(by \eqref{xixi})}\\
&= \bigo_{j=1}^k \left(\beta^{\otimes Q_j} \circ \bigotimes_{i=N_j+1}^{N_{j+1}} \lambda(x_i')\right) \quad\text{(by \eqref{interchange} repeatedly)}\\
&= \left(\beta^{\otimes Q_1}\right)^k \circ \bigo_{j=1}^k \left(\bigotimes_{i=N_j+1}^{N_{j+1}} \lambda(x_i')\right) \quad\text{(by \eqref{betalambda} repeatedly)}\\
&= \left(\beta^{\otimes Q_1}\right)^k \circ \lambda(r') = 0.
\end{split}
\]
The last equality holds because $\lambda(R_{h,\theta}) = 0$. This finishes the proof of the Theorem.
\end{proof}

The next result says that the twisting procedure in Theorem \ref{thm:twist} is compatible with morphisms.

\begin{corollary}
\label{cor0:twist}
Let $\sP = FX/\langle R \rangle$ be a normal PROP and $\theta$ be an arbitrary partition of $I$.  Suppose
\[
\lambda \colon \sP_{h,\theta} \to \sE_V \andspace \lambda' \colon \sP_{h,\theta} \to \sE_{V'}
\]
are Hom-$\sP$-algebras of type $\theta$.  Suppose further that $\beta \colon V \to V$, $\beta' \colon V' \to V'$, and $f \colon V \to V'$ are morphisms of Hom-$\sP$-algebras of type $\theta$ such that
\[
f \circ \beta = \beta' \circ f.
\]
Then $f$ is also a morphism when $V$ and $V'$ are equipped with the Hom-$\sP$-algebra of type $\theta$ structures $\lambda_\beta$ and $\lambda'_{\beta'}$ \eqref{lb}, respectively.
\end{corollary}

\begin{proof}
First note that it suffices to show that
\begin{equation}
\label{flambday}
f^{\otimes q} \circ \lambda_\beta(y) = \lambda'_{\beta'}(y) \circ f^{\otimes p}
\end{equation}
for all $y \in X_{h,\theta}(q,p)$ because $\lambda_\beta$ and $\lambda'_{\beta'}$ are morphisms of PROPs and because of the interchange rule \eqref{interchange}.  Pick an element $y \in X_{h,\theta}(q,p)$.  To prove \eqref{flambday} we compute as follows:
\[
\begin{split}
f^{\otimes q} \circ \lambda_\beta(y) 
&= f^{\otimes q} \circ \beta^{\otimes q} \circ \lambda(y)\\
&= (f \circ \beta)^{\otimes q} \circ \lambda(y)\\
&= (\beta' \circ f)^{\otimes q} \circ \lambda(y)\\
&= \beta'^{\otimes q} \circ f^{\otimes q} \circ \lambda(y)\\
&= \beta'^{\otimes q} \circ \lambda'(y) \circ f^{\otimes p}\\
&= \lambda'_{\beta'}(y) \circ f^{\otimes p}.
\end{split}
\]
This proves \eqref{flambday}.
\end{proof}

\begin{example}
\begin{enumerate}
\item
Theorem \ref{thm:twist} and Corollary \ref{cor0:twist} can be applied to $G$-Hom-associative algebras (section \ref{ex:ghomas}), $n$-ary Hom-Nambu algebras (section \ref{ex:nambu}), (generalized) Hom-bialgebras (section \ref{ex:hombi}), solutions of the Hom-Yang-Baxter equation (section \ref{ex:hybe}), Hom-$A_\infty$-algebras (section \ref{ex:homainf}), Hom-$L_\infty$-algebras (section \ref{ex:homlinf}), and most other Hom-type algebras in the literature.
\item
Theorem \ref{thm:twist} and Corollary \ref{cor0:twist} do \emph{not} apply to Hom-associative algebras of type $II_1$ or type $III$ (section \ref{ex:otherhomas}) because in those cases $S \not= I$.
\end{enumerate}
\end{example}

We now discuss some special cases of Theorem \ref{thm:twist}.

The following result says that, for a normal PROP $\sP$, each multiplicative Hom-$\sP$-algebra gives rise to a derived sequence of multiplicative Hom-$\sP$-algebras whose generating operations are twisted by the twisting map.  Recall the notations and conventions from section \ref{mhomhalg} and Definition \ref{def:mhomalg}.

\begin{corollary}
\label{cor1:twist}
Let $\sP = FX/\langle R \rangle$ be a normal PROP and
\[
\lambda \colon \sP_{mh} \to \sE_V
\]
be a multiplicative Hom-$\sP$-algebra.  Then for each $n \geq 1$, there is another multiplicative Hom-$\sP$-algebra $\lambda_n \colon \sP_{mh} \to \sE_V$ determined by
\[
\begin{split}
\lambda_n(x) &= \left(\lambda(\alpha)^n\right)^{\otimes q} \circ \lambda(x) \quad\text{for $x \in X(q,p)$},\\
\lambda_n(\alpha) &= \lambda(\alpha)^{n+1}.
\end{split}
\]
\end{corollary}

\begin{proof}
Take the trivial partition $\theta_{min}$ \eqref{thetamin} of $I$, and regard $V$ as a Hom-$\sP$-algebra of type $\theta_{min}$ via the composition
\[
\sP_{h,\theta_{min}} \xrightarrow{\pi_1}  \sP_{mh} \xrightarrow{\lambda} \sE_V,
\]
in which $\pi_1$ is the map in Proposition \ref{hompalg}.
The twisting map $\lambda(\alpha) \colon V \to V$ is a morphism of $\sP_{h,\theta_{min}}$-algebras, and hence so is $\lambda(\alpha)^n$ for each $n \geq 1$.  Now apply Theorem \ref{thm:twist} to $\beta = \lambda(\alpha)^n$.  Observe that $\lambda_\beta \colon \sP_{h,\theta_{min}} \to \sE_V$ factors through $\pi_1$ because $\lambda(\alpha)^n$ is compatible with all the $\lambda(x)$ for $x \in X$ \eqref{alphacompat}.  The resulting map $\sP_{mh} \to \sE_V$ is the desired map $\lambda_n$.
\end{proof}

The following result says that, given an algebra $V$ over a normal PROP $\sP$ and a morphism on $V$, there is a corresponding multiplicative Hom-$\sP$-algebra structure on $V$ with that morphism as the twisting map.  Such a twisting construction result for Hom-algebras first appeared in \cite{yau2}, where $G$-Hom-associative algebras were considered.

\begin{corollary}
\label{cor2:twist}
Let $\sP = FX/\langle R \rangle$ be a normal PROP, $\lambda \colon \sP \to \sE_V$ be a $\sP$-algebra, and $\beta \colon V \to V$ be a morphism of $\sP$-algebras.  Then there is a multiplicative Hom-$\sP$-algebra
\[
\lambda_\beta \colon \sP_{mh} \to \sE_V
\]
determined by
\[
\begin{split}
\lambda_\beta(x) &= \beta^{\otimes q} \circ \lambda(x) \quad\text{for $x \in X(q,p)$},\\
\lambda_\beta(\alpha) &= \beta.
\end{split}
\]
\end{corollary}

\begin{proof}
Take the trivial partition $\theta_{min}$ \eqref{thetamin} of $I$, and regard $V$ as a Hom-$\sP$-algebra of type $\theta_{min}$ via the composition
\[
\sP_{h,\theta_{min}} \xrightarrow{\pi}  \sP \xrightarrow{\lambda} \sE_V,
\]
in which $\pi$ is the map in Proposition \ref{hompalg}.  The map $\beta$ is then a morphism of $\sP_{h,\theta_{min}}$-algebras.  Now apply Theorem \ref{thm:twist}.  Observe that $\lambda_\beta \colon \sP_{h,\theta_{min}} \to \sE_V$ factors through $\pi_1 \colon \sP_{h,\theta_{min}} \to \sP_{mh}$ because $\beta$ is a morphism of $\sP$-algebras.  The resulting map $\sP_{mh} \to \sE_V$ is the desired map.
\end{proof}

The next result is another version of Corollary \ref{cor2:twist}.

\begin{corollary}
\label{cor3:twist}
Let $\sP = FX/\langle R \rangle$ be a normal PROP, $\lambda \colon \sP \to \sE_V$ be a $\sP$-algebra, $\beta \colon V \to V$ be a morphism of $\sP$-algebras, and $\theta$ be an arbitrary partition of $I$.  Then there is a Hom-$\sP$-algebra of type $\theta$
\begin{equation}
\label{twistedhomalg}
\lambda_\beta \colon \sP_{h,\theta} \to \sE_V
\end{equation}
determined by
\[
\begin{split}
\lambda_\beta(x) &= \beta^{\otimes q} \circ \lambda(x) \quad\text{for $x \in X(q,p)$},\\
\lambda_\beta(\alpha_p) &= \beta \quad\text{for $p \in \theta$}.
\end{split}
\]
\end{corollary}

\begin{proof}
The desired map $\lambda_\beta$ is the composition of the one in Corollary \ref{cor2:twist} and the map $\pi_1 \colon \sP_{h,\theta} \to \sP_{mh}$ in Proposition \ref{hompalg}.
\end{proof}

In the next two results, we give partial classification of Hom-$\sP$-algebras of the form \eqref{twistedhomalg}.

\begin{corollary}
\label{cor:classification}
Let $\sP = FX/\langle R \rangle$ be a normal PROP and $\theta$ be an arbitrary partition of $I$.  Let $\lambda \colon \sP \to \sE_V$ and $\lambda' \colon \sP \to \sE_{V'}$ be $\sP$-algebras, and let $\beta \colon V \to V$ and $\beta' \colon V ' \to V'$ be morphisms of $\sP$-algebras such that $\beta'^{\otimes q}$ is injective whenever $X(q,p) \not= 0$.  Then the following two statements are equivalent:
\begin{enumerate}
\item
The Hom-$\sP$-algebras of type $\theta$
\[
\lambda_\beta \colon \sP_{h,\theta} \to \sE_V \andspace
\lambda'_{\beta'} \colon \sP_{h,\theta} \to \sE_{V'}
\]
as in \eqref{twistedhomalg} are isomorphic.
\item
There is an isomorphism $\gamma \colon V \to V'$ of $\sP$-algebras such that $\gamma \circ \beta = \beta' \circ \gamma$.
\end{enumerate}
\end{corollary}

\begin{proof}
Assume the first statement is true, and pick an isomorphism $\gamma \colon V \to V'$ of Hom-$\sP$-algebras of type $\theta$.  For any $p \in \theta$, we have
\begin{equation}
\label{gammabetacommute}
\begin{split}
\gamma \circ \beta &= \gamma \circ \lambda_\beta(\alpha_p)\\
&= \lambda'_{\beta'}(\alpha_p) \circ \gamma\\
&= \beta' \circ \gamma.
\end{split}
\end{equation}
To show that $\gamma$ is a morphism of $\sP$-algebras, note that it suffices to show that
\begin{equation}
\label{gammalambdax}
\gamma^{\otimes q} \circ \lambda(x) = \lambda'(x) \circ \gamma^{\otimes p}
\end{equation}
for all $x \in X(q,p)$ because $\lambda$ and $\lambda'$ are morphisms of PROPs and because of the interchange rule \eqref{interchange}.  For $0 \not= x \in X(q,p)$, we compute as follows:
\[
\begin{split}
\beta'^{\otimes q} \circ \gamma^{\otimes q} \circ \lambda(x) &= (\beta' \circ \gamma)^{\otimes q} \circ \lambda(x)\\
&= (\gamma \circ \beta)^{\otimes q} \circ \lambda(x) \quad\text{(by \eqref{gammabetacommute})}\\
&= \gamma^{\otimes q} \circ \beta^{\otimes q} \circ \lambda(x)\\
&= \gamma^{\otimes q} \circ \lambda_\beta(x)\\
&= \lambda'_{\beta'}(x) \circ \gamma^{\otimes p}\\
&= \beta'^{\otimes q} \circ \lambda'(x) \circ \gamma^{\otimes p}.
\end{split}
\]
The above computation and the injectivity assumption on $\beta'^{\otimes q}$ together imply \eqref{gammalambdax}. Therefore, the first statement implies the second one.  The converse is proved by essentially the same argument.
\end{proof}

For a $\sP$-algebra $V$, let $\Aut_{\sP}(V)$ denote the group of $\sP$-algebra automorphisms on $V$.

\begin{corollary}
\label{cor2:classification}
Let $\sP = FX/\langle R \rangle$ be a normal PROP, $\lambda \colon \sP \to \sE_V$ be a $\sP$-algebra, and $\theta$ be an arbitrary partition of $I$.  Then there is a canonical bijection between the following two sets:
\begin{enumerate}
\item
The set of isomorphism classes of Hom-$\sP$-algebras of type $\theta$ of the form $\lambda_\beta$ \eqref{twistedhomalg} in which $\beta$ is a $\sP$-algebra automorphism.
\item
The set of conjugacy classes in the group $\Aut_{\sP}(V)$.
\end{enumerate}
\end{corollary}

\begin{proof}
This is the special case of Corollary \ref{cor:classification} with $V = V'$ and $\lambda = \lambda'$.
\end{proof}

Corollary \ref{cor2:classification} implies that a single  $\sP$-algebra usually gives rise to many isomorphism classes of Hom-$\sP$-algebras of the form \eqref{twistedhomalg}.  For example, when $\sP$ is the PROP for Lie algebras, Corollary \ref{cor2:classification} recovers Corollary 3.2 in \cite{yau6}.  Using this partial classification result, it is observed in \cite{yau6} that the $3$-dimensional Heisenberg algebra, the $(1+1)$-Poincar\'{e} algebra, and $\sltwo$ each has uncountably many associated isomorphism classes of Hom-Lie algebras.

%%==============%%
%%              %%
%%  References  %%
%%              %%
%%==============%%

\end{document}